\documentclass[12pt]{article}

\usepackage[centertags]{amsmath}
\usepackage{amsfonts}
\usepackage{amssymb}
\usepackage{amsthm}
\usepackage{newlfont}
\usepackage{amscd}
\newlength{\defbaselineskip}
\newcommand{\setlinespacing}[1]%
           {\setlength{\baselineskip}{#1 \defbaselineskip}}

\theoremstyle{plain}
\newtheorem{thm}{Theorem}[section]
\newtheorem{cor}[thm]{Corollary}
\newtheorem{lem}[thm]{Lemma}

\theoremstyle{definition}
\newtheorem{defn}[thm]{Definition}
\newtheorem{rem}[thm]{Remark}

\numberwithin{equation}{section}

\begin{document}

\title{\Large\textbf{On Baer Invariants of Pairs of Groups}}

\author{\textbf{Zohreh Vasagh, Hanieh Mirebrahimi}\\ \textbf{and Behrooz Mashayekhy} \\
Department of Pure Mathematics,\\
Center of Excellence in Analysis on Algebraic Structures,\\
Ferdowsi University of Mashhad,\\
P.O.Box 91775-1159, Mashhad, Iran.\\
E-mail: zo$_{-}$va761@stu-mail.um.ac.ir,\\ h$_{-}$mirebrahimi@um.ac.ir, \\
bmashf@um.ac.ir}
\date{ }
\maketitle
%----------------------------------------------------------------------------------------------%
\begin{abstract}
In this paper, we use the theory of simplicial groups to develop the Schur multiplier of a pair of groups $(G,N)$ to the Baer invariant of it, $\mathcal{V}M(G,N)$, with respect to an arbitrary variety $\mathcal{V}$. Moreover, we present among other things some behaviors of Baer invariants of a pair of groups with respect to the free product and the direct limit.
Finally we prove that the nilpotent multiplier of a pair of groups does commute with the free product of finite groups of mutually coprime orders.
\end{abstract}

%----------------------------------------------------------------------------------------------%
\section{Introduction and motivation}
Schur (1904), introduced the Schur multiplier of a group $G$, $M(G)$, by Projective representations of $G$ which the second integral homology group of $G$. The second homology plays the special role, this led H. Hopf (1942), to find an effective method for calculating it. Hopf's integral homology formula is identical to the Schur multiplier of a finitely presented group. R. Baer (1945) using the variety of groups, generalized the notion of the Schur multiplier of a group $G$ to the Baer invariant of it with respect to a variety $\mathcal{V}$, $\mathcal{V}M(G)$.

 In the present section, we outline the topological interpretations of the Schur multiplier and the Schur multiplier of a pair of groups, and a motivation to define the Baer invariant of a pair of groups with topological approach. Proofs, requiring the theory of simplicial groups, deferred to Section 2. Section 3 is devoted to the results that obtained from the long exact sequence of the Baer invariant of a pair of groups. Also we show that our definition of the Baer invariant of a pair of groups is a vast generalization of the Baer invariant of a special pair of groups which was defined by Moghaddam, Salemkar and Saany (2007). At the end of Section 3, we show that the Baer invariant of a pair of groups commutes with the direct limit. Also we obtain a long exact sequence that contains the Schur multiplier and the 2-nilpotent multiplier of a pair of groups. Computation of the Baer invariant of a pair of free products of groups is given in Section 4. Also in this section, we present an explicit formula for the 2-nilpotent multiplier of a pair of free products of groups.

 As convention, throughout the article we use $\mathcal{V}$ as an arbitrary variety of groups defined by a set of laws $V$. We note that for any group $G$ one can construct functorially a free simplicial group $K_.$, called \emph{free simplicial resolution} of $G$, whose $\pi_0(K_.)\cong G$, $\pi_m(K_.)\cong 0$ for $m\geq 1$, with $K_m$ is free group (see Duskin, 1975). Given a functor $T:(Groups)\rightarrow (Groups)$ we define left derived functors as
$$ L^T_m(G)= \pi_m\big(T(K_.)\big), \ \ m\geq 0.$$
The groups $L^T_m(G)$ are independent of the choice of the free simplicial resolution. For more details see for instance Inassaridze (1974).

The topological interpretation of the c-nilpotent multiplier of a group arose in the work of Burns and Ellis (1997), on the 2-nilpotent multiplier of the free product of groups. Burns and Ellis (1997) observed that there are natural isomorphisms
$$\frac{G}{\gamma_{c+1}(G)}\cong L^{\tau_c}_0(G)\cong\pi_0(\frac{K_.}{\gamma_{c+1}(K_.)})$$
  $$M^{(c)}(G)\cong L^{\tau_c}_1(G)\cong\pi_1(\frac{K_.}{\gamma_{c+1}(K_.)}),$$
where $K_.$ is a free simplicial resolution of $G$ and $\tau_c$ is the functor that sends $G$ to $G/\gamma_{c+1}(G)$. In Franco (1998), described the Baer invariant of a group $G$ in topological language as follows:
$$\frac{G}{V(G)}\cong L^{\tau_V}_0(G)\cong\pi_0(\frac{K_.}{V(K_.)})$$
 $$\mathcal{V}M(G)\cong L^{\tau_V}_1(G)\cong\pi_1(\frac{K_.}{V(K_.)}),$$
where $K_.$ is a free simplicial resolution of $G$ and $\tau_V$ is the functor that sends $G$ to $G/V(G)$ .

A group $G$ with a normal subgroup $N$ denoted by  $(G,N)$ is called a homomorphism of pairs $(G_1,N_1) \rightarrow~(G_2,N_2)$ is a group homomorphism $G_1 \rightarrow G_2$ that sends $N_1$ into $N_2$.

Ellis (1998), introduced the Schur multiplier of a pair $(G,N)$ as a functorial abelian group $M(G,N)$ whose feature is the following natural exact sequence
\begin{equation}
\begin{array}{ll}
\cdots&\rightarrow M(G,N)\rightarrow M(G)\rightarrow M(\frac{G}{N})\\
&\rightarrow \frac{N}{[N,G]}\rightarrow (G)^{ab} \rightarrow (\frac{G}{N})^{ab}\rightarrow 0.
\end{array}
\end{equation}

The natural epimorphism $G\rightarrow G/N$ implies the following exact sequence of free simplicial groups
$$1\rightarrow\ker (\alpha) \rightarrow K_. \stackrel{\alpha}{\rightarrow}L_.\rightarrow 1,$$
where $K_.$ and $L_.$ are free simplicial resolution of $G$ and $G/N$, respectively (see Franco, 1998). The short exact sequence of simplicial groups
\begin{equation}\label{valpha}1\rightarrow\ker (\frac{\alpha}{V(\alpha)}) \rightarrow \frac{K_.}{V(K_.)} \stackrel{\frac{\alpha}{V(\alpha)}}{\rightarrow} \frac{L_.}{V(L_.)}\rightarrow 1\end{equation}
gives rise a long exact sequence of homotopy groups as follows:
$$\begin{array}{ll}\cdots &\rightarrow\pi_1(\ker \frac{\alpha}{V(\alpha)}) \rightarrow \pi_1(\frac{K_.}{V(K_.)}) {\rightarrow} \pi_1(\frac{L_.}{V(L_.)})\\&\rightarrow \pi_{0}(\ker \frac{\alpha}{V(\alpha)}) \rightarrow \pi_{0}(\frac{K_.}{V(K_.)}) {\rightarrow} \pi_{0}(\frac{L_.}{V(L_.)})\rightarrow 0.\end{array}$$
Franco (1998), proved that $\pi_0(\ker \frac{\alpha}{V(\alpha)})\cong \frac{N}{[NV^*G]}$. Using some isomorphisms we can rewrite the above long exact sequence as follows:
\begin{equation}\label{long}\begin{array}{ll}\cdots &\rightarrow\pi_1(\ker \frac{\alpha}{V(\alpha)}) \rightarrow \mathcal{V}M(G){\rightarrow} \mathcal{V}M(G/N)\\&\rightarrow \frac{N}{[NV^*G]} \rightarrow \frac{G}{V(G)} {\rightarrow} \frac{G/N}{V(G/N)}\rightarrow 0.\end{array}\end{equation}
Indeed Franco obtained the Fr\"{o}hlich long exact sequence independent of the method of Fr\"{o}hlich (1963).

 Now, we define the Baer invariant of a pair of groups $(G,N)$ as follows:
  $$\mathcal{V}M(G,N)=\pi_1\big(\ker( \frac{\alpha}{V(\alpha)})\big).$$
%--------------------------------------------------------------------------------------%
\section{Preliminaries and notation}
 In this section we recall some basic notations and properties of simplicial groups which will be needed in the sequel. We refer the reader to Curtis (1971) or Georss and Jardine (1999) for further details.
%---------------------------------------------------------------------------------------%
\begin{defn}
A \emph{simplicial set} $K_.$ is a sequence of sets $K_0, K_1, K_2,\ldots$ together with maps $d_i:K_n\rightarrow K_{n-1}$ (faces) and $s_i:K_n\rightarrow K_{n+1}$ (degeneracies), for each $0\leq i\leq n$, such that the following conditions hold:
\[ \begin{array}{lcl}
  d_j d_i  &= & \ \ \ d_{i-1} d_j \ \ \quad\text{for $ j<i$} \\
   s_j s_i & =& \ \ \ s_{i+1} s_j \ \ \quad\text{for $ j\leq i$}\\
d_j s_i&=&\begin{cases}
s_{i-1} d_j  &\text{for $ j<i$;} \\
identity &\text{for $ j=i,i+1$;} \\
s_i d_{j-1}&\text{for $ j>i+1$.} \\
\end{cases} \end{array}\]
A \emph{simplicial map} $f:K_.\rightarrow L_.$ is a sequence of functions $f_n:K_n\rightarrow L_n$, with the following commutative diagram
\[
\begin{CD}
K_{n+1}    @<s_i<<  {K_n}  @>{d_i}>>  K_{n-1}\\
@Vf_{n+1}VV     @Vf_{n}VV         @VVf_{n-1}V\\
L_{n+1}    @<s_i<<  {L_n}  @>d_i>>  L_{n-1}
\end{CD}
\]
\end{defn}
%---------------------------------------------------------------------------------------%
Like topological spaces, the homotopy groups of simplicial sets is defined.
The category of simplicial sets and topological spaces can be related by two functors as follows:
\begin{itemize}
\item The \emph{geometric realization}, $|-|$, is the functor from the category of simplicial sets to the category of CW complexes.
\item The \emph{singular simplicial}, $S_*(-)$, is the functor from the category of topological spaces to the category of simplicial sets.
        \end{itemize}

A simplicial set $K_.$ is called a \emph{simplicial group} if each $K_i$ is group and all faces and degeneracies are homomorphisms.
There is a basic property of simplicial groups which due to Moore (1954--55) its homotopy groups $\pi_*(G_.)$ can be obtained as the 'homology of' a certain chain complex $(NG,\partial)$.
%---------------------------------------------------------------------------------------%
 \begin{defn}
If  $K_.$ is a simplicial group, then the \emph{Moore complex} $(NK_.,\partial)$ of $K_.$ is the (nonabelian) chain complex defined by $(NK)_n=\cap_{i=0}^{n-1}Ker d_i$ with $\partial_n:NK_n\rightarrow NK_{n-1}$ which is the restriction of $d_n$.
\end{defn}
%---------------------------------------------------------------------------------------%
A simplicial group $K_.$ is said to be \emph{free} if each $K_n$ is a free group and degeneracy homomorphisms $s_i$'s send the free basis of $K_n$ into the free basis of $K_{n+1}$.

%---------------------------------------------------------------------------------------%
\begin{thm}\label{sg} (see Curtis, (1971)).
\begin{enumerate}
\item For every simplicial group $K_.$ the homotopy group $\pi_n(K_.)$ is abelian even for $n=1$.\label{a}
\item Every epimorphism  between simplicial groups is a fibration.\label{f}
\item  Let $K_.$ be a simplicial group, then $\pi_*(K_.)\cong H_*(NK_.)$.\label{m}
\item $H_n\big(N(K_.\otimes L_.)\big)\cong H_n\big(N(K_.)\otimes N(L_.)\big).$ \label{n}
\end{enumerate}
\end{thm}
%---------------------------------------------------------------------------------------%

%-------------------------------------------------------------------------------------
\section{Some properties of the Baer invariant of a pair of groups}
 In this section we study some behaviors of the Baer invariant of a pair of groups.
 Let $f:(G_1,N_1)\rightarrow(G_2,N_2)$ be a homomorphism of pairs of groups, then functorial property of free simplicial resolution yields the following diagram of free simplicial groups
$$
\begin{CD}\ker(\alpha_1)@>>>{K_{1_.}} @>\alpha_1>>{L_{1_.}}\\@VVV@VVV@VVV\\ \ker(\alpha_2) @>>> {K_{2_.}}@>{\alpha_2}>> {L_{2_.}},\end{CD}
$$
where $K_{i_.}$ and $L_{i_.}$ are the corresponding free simplicial resolution of $G_i, G_i/N_i$, respectively. Therefore we have the following commutative diagram
$$
\begin{CD}\ker\alpha_1/V(\alpha_1)@>>> K_{1_.}/V(K_{1_.}) @>\alpha_1/V(\alpha_1)>> L_{1_.}/V(L_{1_.})\\@V\gamma VV@VVV@VVV\\ \ker\alpha_2/V(\alpha_2)@>>> K_{._2}/V(K_{._2}) @>\alpha_2/V(\alpha_2)>> L_{._2}/V(L_{._2}).\end{CD}$$
By the above diagram we have $\pi_1(\gamma):\pi_1\big(\ker\alpha_1/V(\alpha_1)\big)\rightarrow\pi_1\big(\ker\alpha_2/V(\alpha_2)\big)$.
Indeed, we can state the following theorem.
\begin{thm}
The Baer invariant of a pair of groups is a functor from the category of pairs of groups to the category of abelian groups.
\end{thm}

 The long exact sequence of \eqref{long} implies the following theorems.
 \begin{thm}\label{ps}
The Baer invariant of a group is a special case of the Baer invariant a pair of groups i.e. $\mathcal{V}M(G,G)\cong \mathcal{V}M(G)$. Thus for a cyclic group $C$ and a free group $F$ we have $\mathcal{V}M(C,C)=1=\mathcal{V}M(F,F)$. Also $\mathcal{V}M(G,1)$ is a trivial group.
\end{thm}
\begin{thm}
Let $G$ be the semi-direct product of $N$ by $Q$. Then $\mathcal{V}M(G,N)\cong\ker\big(\mathcal{V}M(G)\twoheadrightarrow\mathcal{V}M(G/N)\big)$
and $\mathcal{V}M(G)\cong\mathcal{V}M(G,N)\oplus\mathcal{V}M(Q).$
\end{thm}
\begin{proof}
The hypothesis implies that the exact sequence \eqref{valpha} splits and hence the result holds.
\end{proof}
The above theorem shows that if $G$ is the semi-direct product of $N$ by $Q$, then the Baer invariant of $(G,N)$ can be described in presentation of groups as follows.
\begin{cor}
Let $G\cong F/R$ be a free presentation of $G$ and $N\cong S/R$ be a nomal subgroup of $G$ which has a complement in $G$, then  $$\mathcal{V}M(G,N)\cong \frac{R\cap[SV^*F]}{[RV^*F]}.$$
\end{cor}
Note that the above corollary shows that our definition of the Baer invariant of a pair of groups is a vast generalization of the one by Moghaddam, Salemkar and Saany (2007).
\begin{thm}
Suppose that $M$ and $N$ are two subgroups of a group $G$ such that $M\cong MN$, then there exists the following isomorphism  $$\mathcal{V}M(MN,N)\cong \mathcal{V}M(M,M\cap N).$$
\end{thm}
\begin{proof}
By the second isomorphism theorem we have $\frac{MN}{N}\cong \frac{M}{M\cap N}$. Let $K_.$ and $L_.$ be the free simplicial groups corresponding to $MN$ and $MN/N$, respectively. Because of the functorial property of
free simplicial resolution corresponding to each group, we conclude that $K_.$ and $L_.$ are also simplicial groups corresponding to $M$ and $\frac{M}{M\cap N}$, respectively. Hence by the definition the result holds.
\end{proof}

Using the exact sequence (1.3) and the structure of its sixth term given by Eckmann, Hilton, and Stammbach (1972) and Lue (1976), when $N$ is a central and an $\mathcal{N}_c$-central subgroup of $G$, respectively, we have the following theorem.
\begin{thm}
Let $N$ be a central subgroup of $G$ then $M(G,N)\cong G^{ab}\otimes N$. Also, if $N$ is an $\mathcal{N}_c$-central subgroup of $G$, then $M^{(c)}(G,N)\cong N \otimes \frac{G}{\gamma_c(G)} \otimes \cdots \otimes \frac{G}{\gamma_c(G)}$ with $c$ copies of $\frac{G}{\gamma_c(G)}$
\end{thm}
\begin{thm}
Let $\{(G_i,N_i)\}_{i\in I}$ be a given direct system of pairs
of groups with the directed index set $I$, then ${\mathcal{V}M(\varinjlim G_i,\varinjlim N_i)\cong \varinjlim \mathcal{V} M(G_i,N_i)}.$
\end{thm}
\begin{proof}
For any $i\in I$, let $K_{i_.}$ and $L_{i_.}$ be the corresponding free simplicial resolutions of $G_i$ and $G_i/N_i$, respectively. Assume that $\alpha_i:K_{i_.}\rightarrow L_{i_.}$ is the corresponding epimorphism of simplicial groups. Thus we can consider the following exact sequence of simplicial groups.
$$1\rightarrow\ker (\frac{\alpha_i}{V(\alpha_i)})\rightarrow \frac{K_{i_.}}{V(K_{i_.})} \stackrel{\frac{\alpha_i}{V(\alpha_i)}}{\rightarrow} \frac{L_{i_.}}{V(L_{i_.})}\rightarrow 1.$$
Vasagh, Mirebrahimi and Mashayekhy proved that $\varinjlim \pi_n(K_{i_.})\cong \pi_n(\varinjlim K_{i_.})$, where $K_{i_.}$ is a simplicial group. Hence $\varinjlim K_{i_.}$ and $\varinjlim L_{i_.}$ are simplicial groups corresponding to $\varinjlim G_i$ and $\varinjlim G_i/N_i$, respectively, and we have the following exact sequence:
$$1\rightarrow\ker\big(\frac{\varinjlim(\alpha_i)}{V\big(\varinjlim(\alpha_i)\big)}\big) \rightarrow \frac{\varinjlim(K_{i_.})}{V\big(\varinjlim(K_{i_.})\big)} \stackrel{\frac{\varinjlim(\alpha_i)}{V\big(\varinjlim(\alpha_i)\big)}}{\rightarrow} \frac{\varinjlim(L_{i_.})}{V\big(\varinjlim(L_{i_.})\big)}\rightarrow 1$$
Since the functor $-/V(-)$ has right adjoint, $\varinjlim (\frac{K_{i_.}}{V(K_{i_.})})\cong  \frac{\varinjlim K_{i_.}}{V(\varinjlim K_{i_.})}$. The fact that direct limit preserve the exact sequence yields the following commutative diagrams
$$\begin{array}{ccccccccc}1 & \rightarrow &\ker(\frac{\varinjlim(\alpha_i)}{V(\varinjlim(\alpha_i))}) &\rightarrow & \frac{\varinjlim(K_{i_.})}{V\big(\varinjlim(K_{i_.})\big)}& {\rightarrow}& {\frac{\varinjlim(L_{i_.})}{V\big(\varinjlim(L_{i_.})\big)}}&\rightarrow& 1\\
&&\downarrow&&\downarrow&&\downarrow&&\\1&\rightarrow&\varinjlim\big(\ker (\frac{\alpha_i}{V(\alpha_i)})\big)&\rightarrow &\varinjlim \big(\frac{K_{i_.}}{V(K_{i_.})}\big) &{\rightarrow} &\varinjlim \big(\frac{L_{i_.}}{V(L_{i_.})}\big)&\rightarrow &1.\end{array}$$

Five Lemma implies  that $\varinjlim (\ker (\frac{\alpha_i}{V(\alpha_i)}))\cong  \ker \varinjlim(\frac{\alpha_i}{V(\alpha_i)})$.
Also the homotopy groups of simplicial groups commute with direct limits (See Vasagh, Mirebrahimi and Mashayekhy), hence we have
  $$\mathcal{V}M(\varinjlim G_i,\varinjlim N_i)\cong \pi_1\big( \ker \varinjlim(\frac{\alpha_i}{V(\alpha_i)})\big)$$
  $$\cong\pi_1\Big(\varinjlim \big(\ker (\frac{\alpha_i}{V(\alpha_i)})\big)\Big)\cong  \varinjlim \mathcal{V} M(G_i,N_i).$$
\end{proof}
%--------------------------------------------------------------------------------------------------%
Let $G$ be a group and $N$ be a normal subgroup of it, and consider $K_.$ and $L_.$ as the free simplicial resolutions corresponding to $G$ and $G/N$, respectively. The simplicial epimorphism $\alpha:K_.\rightarrow L_.$ gives rise epimorphisms $\alpha_n:K_./\gamma_n(K_.)\rightarrow L_./\gamma_{n}(L_.)$ and $\beta_n:\gamma_{n}(K_.)/\gamma_{n+1}(K_.)\rightarrow \gamma_{n}(L_.)/\gamma_{n+1}(L_.)$ which induce the following commutative diagram:
\begin{equation}\label{dian}\begin{array}{cccccccccc}
 &&1&&1&&1\\&&\downarrow&&\downarrow&&\downarrow&&\\1&\rightarrow& \ker(\beta_n)&\rightarrow &\frac{\gamma_{n}(K_.)}{\gamma_{n+1}(K_.)}&{\rightarrow}&\frac{ \gamma_{n}(L_.)}{\gamma_{n+1}(L_.)}&\rightarrow&1\\
 &&\downarrow&&\downarrow&&\downarrow&&\\ 1&\rightarrow&\ker(\alpha_{n+1})&\rightarrow& \frac{K_.}{\gamma_{n+1}(K_.)} &{\rightarrow}&\frac{L_.}{\gamma_{n+1}(L_.)}&\rightarrow&1\\&&\downarrow&&\downarrow&&\downarrow&&\\1& \rightarrow&\ker(\alpha_n)&\rightarrow&
 \frac{K_.}{\gamma_{n}(K_.)}&{\rightarrow}&
 \frac{L_.}{\gamma_{n}(L_.)}&\rightarrow&1\\&&\downarrow&&\downarrow&&\downarrow&&\\&&1&&1&&1
 \end{array}\end{equation}
 Since $\beta_n$ is an epimorphism so is $\ker(\alpha_{n+1})\rightarrow\ker(\alpha_{n})$. Since every epimorphism of simplicial groups is a fiberation, the left column exact sequence induces the following long exact sequence of homotopy groups
$$\begin{array}{ll}\cdots&\rightarrow\pi_1\big(\ker(\beta_{n})\big)\rightarrow\pi_1\big(\ker(\alpha_{n+1})\big)
\rightarrow \pi_1\big(\ker(\alpha_{n})\big)\\&\rightarrow \pi_0\big(\ker(\beta_{n})\big)\rightarrow\pi_0\big(\ker(\alpha_{n+1})\big) \rightarrow \pi_0\big(\ker(\alpha_{n})\big)\rightarrow 1.\end{array}$$
Using some isomorphisms we can rewrite the above sequence as the following long exact sequence:
\begin{equation}\label{eq}\begin{array}{ll}\cdots&\rightarrow\pi_1\big(\ker(\beta_{n})\big)\rightarrow M^{(n)}(G,N) \rightarrow M^{(n-1)}(G,N)\\ &\rightarrow \pi_0\big(\ker(\beta_{n})\big)\rightarrow \frac{N}{\gamma_{n+1}(G,N)}\rightarrow \frac{N}{\gamma_{n}(G,N)}\rightarrow 1.\end{array}\end{equation}

Now for $n=2$, we discuss on the long exact sequence \eqref{eq}.
First in the following lemma, we concentrate on $\ker(\beta_2)$ which gives a relation between exterior product of a group and its quotient.
\begin{lem}
Let $(G,N)$ be a pair of group, then we have the following exact sequence of groups
 $$1\rightarrow \frac{N}{[N,G]}\wedge\frac{N}{[N,G]}\oplus \frac{N}{[N,G]}\otimes (\frac{G}{N})^{ab}\rightarrow G^{ab}\wedge G^{ab}\rightarrow (\frac{G}{N})^{ab}\wedge (\frac{G}{N})^{ab}\rightarrow 1.$$
\end{lem}
\begin{proof}
 It is known that if $F$ is a free group, then $\gamma_2(F)/\gamma_3(F)$ is a free abelian group which is isomorphic to $F^{ab}\wedge F^{ab}$. Hence computing the ranks of the terms of the  first row exact sequence of commutative diagram \eqref{dian} implies that $$\ker(\beta_{2})\cong (\ker(\alpha_2)\wedge\ker(\alpha_2))\oplus (\ker(\alpha_2)\otimes L_.^{ab}).$$
This row exact sequence yields the following long exact sequence of homotopy groups as follows:
$$\begin{array}{llll} \cdots&\rightarrow \pi_1(\ker(\beta_{2})&\rightarrow\pi_1(K_.^{ab}\wedge K_.^{ab})\rightarrow\pi_1(L_.^{ab}\wedge L_.^{ab})\\&\rightarrow\pi_0(\ker(\beta_{2})&\rightarrow\pi_0(K_.^{ab}\wedge K_.^{ab})\rightarrow\pi_0(L_.^{ab}\wedge L_.^{ab})\rightarrow 1.\end{array}
$$
Since $\gamma_2(L_.)/\gamma_3(L_.)$ is a free abelian group so every right term of above long exact sequence splits.
 Burns and Ellis, (1997) showed that $$\pi_0(K_.^{ab}\wedge K_.^{ab})\cong G^{ab}\wedge G^{ab}.$$
 Also we have $$\begin{array}{lll}\pi_0(K_.^{ab}\otimes K_.^{ab})&\cong H_0\big(N(K_.^{ab}\otimes K_.^{ab})\big)\ \ \ & ( by\ Theorem\ \ref{sg} \eqref{m})\\ & \cong H_0\big(N(K_.^{ab})\otimes N(K_.^{ab})\big)& (by\ Theorem\ \ref{sg} \eqref{n})\\ & \cong  H_0\big(N(K_.^{ab})\big)\otimes H_0\big(N(K_.^{ab})\big)&(by\ Kunneth\ formula)\\ & \cong  \pi_0(K_.^{ab})\otimes \pi_0( K_.^{ab})&(by\ Theorem\ \ref{sg} \eqref{m})\\&\cong G^{ab}\otimes G^{ab}.\end{array}$$
 Therefore we have $$\begin{array}{ll}\pi_0\big(\ker(\beta_{2})\big)&\cong\Big(\pi_0\big(\ker(\alpha_2)\big)\wedge
 \pi_0\big(\ker(\alpha_2)\big)\Big)
\oplus \pi_0\big(\ker(\alpha_2)\otimes L_.^{ab}\big)\\ &\cong(\frac{N}{[N,G]}\wedge\frac{N}{[N,G]})\oplus (\frac{N}{[N,G]}\otimes (\frac{G}{N})^{ab}). \end{array}$$
\end{proof}
\begin{thm}
For each group $G$ and a normal subgroup $N$, there exists a functorial group $V(G,N)$ which fits into the following natural exact sequence
$$\begin{array}{ll}&V(G,N)\oplus Tor(\frac{N}{[N,G]},\frac{G}{N}) \oplus (M^{(1)}(G,N)\otimes (\frac{G}{N})^{ab})\oplus (\frac{N}{[N,G]}\otimes M^{(1)}(\frac{G}{N}))\\ &\rightarrow M^{(2)}(G,N) \rightarrow M^{(1)}(G,N) \rightarrow \frac{N}{[N,G]}\wedge\frac{N}{[N,G]}\oplus \frac{N}{[N,G]}\otimes (\frac{G}{N})^{ab}\\&\rightarrow \frac{N}{\gamma_{3}(G,N)}\rightarrow \frac{N}{[G,N]}\rightarrow 1\end{array}$$
\end{thm}
\begin{proof}
It sufficient to compute $\pi_1\big(\ker(\beta_{2})\big)$ and replace it in the long exact sequence \eqref{eq}.
  Similar to the proof of the above lemma the K\"{u}nneth formula implies that  $$\begin{array}{ll}\pi_1\big(\ker(\beta_{2})\big)&\cong \pi_1\big(\ker(\alpha_2)\wedge\ker(\alpha_2)\big) \oplus H_1\Big(N\big(\ker(\alpha_2)\big)\otimes N(L^{ab})\Big)\\ &\cong \pi_1\big(\ker(\alpha_2)\wedge\ker(\alpha_2)\big) \oplus M(G,N)\otimes (\frac{G}{N})^{ab}\\&\oplus \frac{N}{[N,G]}\otimes M(\frac{G}{N})\oplus Tor(\frac{N}{[N,G]},\frac{G}{N}).\end{array}$$
   The group $V(G,N)$ is defined as an abelian group $ \pi_1\big(\ker(\alpha_2)\wedge\ker(\alpha_2)\big).$
\end{proof}

Note that if we put $N=G$ in the above theorem, then we have the natural exact sequence which is proved in Burns and Ellis (1997).

%---------------------------------------------------------------------------------------------------%
\section{The Baer invariant of a pair of the free product of groups}
In this section we study the behavior of the Baer invariant of a pair of the free product of groups.
For $i=1,2$, assume that $\alpha_i:K_{i_.}\rightarrow L_{i_.}$ are epimorphisms, where $K_{i_.}$ and $L_{i_.}$ are free simplicial resolutions corresponding to $G_i$ and $G_i/N_i$, respectively. Van-Kampen theorem for simplicial groups implies that $K_{1_.}\ast K_{2_.}$ and $L_{1_.}\ast L_{2_.}$ are the free simplicial groups corresponding to $G_1\ast G_2$ and $(G_1/N_1\ast G_2/N_2)\cong\frac{G_1*G_2}{\langle N_1\ast N_2\rangle^{G_1\ast G_2}}$, respectively. So we can consider $\beta=~\alpha_1\ast\alpha_2$ as the corresponding epimorphism from $K_{1_.}\ast K_{2_.}$ onto $L_{1_.}\ast L_{2_.}$ (see Burns and Ellis, 1997).

Consider the following exact sequences of simplicial groups
\[1\rightarrow\ker\big(\alpha_1/V(\alpha_1)\big)\rightarrow \frac{K_{1_.}}{V(K_{1_.})} \stackrel{\alpha_1/V(\alpha_1)}{\rightarrow}\frac{L_{1_.}}{V(L_{1_.})}\rightarrow1,\]

\[1\rightarrow\ker\big(\alpha_2/V(\alpha_2)\big)\rightarrow \frac{K_{2_.}}{V(K_{2_.})}\stackrel{\alpha_2/V(\alpha_2)}{\rightarrow} \frac{L_{2_.}}{V(L_{2_.})}\rightarrow 1,\]
\begin{equation}\label{free}1\rightarrow\ker\big(\beta/V(\beta)\big)\rightarrow \frac{K_{1_.}\ast K_{2_.}}{V(K_{1_.}\ast K_{2_.})}\stackrel{\beta/V(\beta)}{\rightarrow} \frac{L_{1_.}\ast L_{2_.}}{V(L_{1_.}\ast L_{2_.})}\rightarrow 1.\end{equation}
Therefore we have the following commutative diagram
\begin{equation}\tiny{\label{dia}\begin{array}{cccccccccc}
 &&1&&1&&1\\&&\downarrow&&\downarrow&&\downarrow&&\\1&\rightarrow& \ker(\theta)&\rightarrow &\ker(\varphi_V)&\stackrel{\theta}{\rightarrow}& \ker(\psi_V)&\rightarrow&1\\
 &&\downarrow&&\downarrow&&\downarrow&&\\ 1&\rightarrow&\ker(\frac{\beta}{V(\beta)})&\rightarrow& \frac{K_{1_.}\ast K_{2_.}}{V(K_{1_.}\ast K_{2_.})} &\stackrel{\beta/V(\beta)}{\rightarrow}&\frac{L_{1_.}\ast L_{2_.}}{V(L_{1_.}\ast L_{2_.})}&\rightarrow&1\\&&\varphi_V|\downarrow\ \ \ \ \ \ &&{\varphi_V}\downarrow \ \ \ \ \ &&\psi_V\downarrow\ \ \ \ \ &&\\1& \rightarrow&\ker(\frac{\alpha_1}{V(\alpha_1)})\times\ker(\frac{\alpha_2}{V(\alpha_2)})&\rightarrow&
 \frac{K_{1_.}}{V(K_{1_.})}\times\frac{K_{2_.}}{V(K_{2_.})}&\stackrel{(\alpha_1,\alpha_2)}{\rightarrow}&
 \frac{L_{1_.}}{V(L_{1_.})}\times\frac{L_{2_.}}{V(L_{2_.})}&\rightarrow&1\\&&\downarrow&&\downarrow&&\downarrow&&\\
 &&1&&1&&1
 \end{array}}\end{equation}

Definition of $\varphi_V$ implies that $\varphi_V|$ is an epimorphism, and the epimorphism $\varphi_V|$ yields the epimorphism $\theta$.
Since every epimorphism of simplicial groups is a fiberation, the left column exact sequence of \eqref{dia} induces the following long exact sequence of abelian groups
\begin{equation}\label{k}
\cdots\rightarrow \pi_n\big((\ker(\theta)\big)\rightarrow\pi_n\big(\ker \frac{\beta}{V(\beta)}\big)\rightarrow\pi_n \big(\ker\frac{\alpha_1}{V(\alpha_1)}\big)\oplus\pi_n\big(\ker\frac{\alpha_2}{V(\alpha_2)}\big)\rightarrow \cdots
\end{equation}
Consider the natural homomorphisms $\iota_i:K_{i_.}/V(K_{i_.})\rightarrow K_{1_.}\ast K_{2_.}/V(K_{1_.}\ast~K_{2_.})$ for $i=1,2$. Since $\pi_n\big(K_{1_.}/V(K_{1_.})\big)\oplus\pi_n\big({K_{2_.}}/{V(K_{2_.})}\big)$ is a coproduct in the category of abelian groups for all $n>0$, there exists
 $${\delta_V}_n: \pi_n\big(K_{1_.}/{V(K_{1_.})}\big)\oplus\pi_n\big(K_{2_.}/{V(K_{2_.})}\big)\rightarrow\pi_n\big(K_{1_.}\ast~K_{2_.}/V(K_{1_.}\ast K_{2_.})\big)$$ such that $ {\varphi_V}_n\circ\pi_n(\delta_V)=id$. Similarly we have
 $$\tau_i:L_{i_.}/V(L_{i_.})\rightarrow L_{1_.}\ast~L_{2_.}/V(L_{1_.}~\ast~L_{2_.}),$$ for $i=1,2$, therefore there exists  $${\sigma_V}_n:\pi_n\big(L_{1_.}/{V(L_{1_.})}\big)\oplus\pi_n\big(L_{2_.}/{V(L_{2_.})}\big)\rightarrow\pi_n\big(L_{1_.}
 \ast L_{2_.}/V(L_{1_.}\ast L_{2_.})\big)$$ such that $ {\psi_V}_n\circ\pi_n(\sigma_V)=id$.

  The following commutative diagram
$$\begin{array}{ccccccc}
    \frac{K_{1_.}\ast K_{2_.}}{V(K_{1_.}\ast K_{2_.})} &\stackrel{\beta/V(\beta)}{\rightarrow}&\frac{L_{1_.}\ast L_{2_.}}{V(L_{1_.}\ast L_{2_.})}&\rightarrow&1\\ \uparrow&&\uparrow&&\\
 \frac{K_{i_.}}{V(K_{i_.})}&\stackrel{\alpha_i}{\rightarrow}&
 \frac{L_{i_.}}{V(L_{i_.})}&\rightarrow&1
 \end{array}$$
 gives rise $\iota_i|:\ker \big(\alpha_i/V(\alpha_i)\big) \rightarrow \ker\big(\beta/V(\beta)\big)$. Similarly there exists $$\eta:\pi_n\Big(\ker\big({\alpha_1}/{V(\alpha_1)}\big)\Big)\oplus\pi_n\Big(\ker\big({\alpha_2}/{V(\alpha_2)}\big)\Big)
 \rightarrow \pi_n\Big(\ker\big(\beta/V(\beta)\big)\Big)$$ such that $ {\varphi_V|}_n\circ\pi_n(\delta_V|)=id.$\\
Consequently, for all $n>0$, the exact sequence of \eqref{k} splits, therefore
\begin{equation}\label{ker}
\pi_n \big(\ker(\theta)\big)\oplus\pi_n \big(\ker\frac{\alpha_1}{V(\alpha_1)}\big)\oplus\pi_n\big(\ker\frac{\alpha_2}{V(\alpha_2)}\big)\cong\pi_n\big(\ker \frac{\beta}{V(\beta)}\big).
\end{equation}
For $n=1$, using some isomorphisms, we can rewrite \eqref{ker} as follows:
$$ \mathcal{V}M(G_1*G_2,\langle N_1\ast N_2\rangle^{G_1\ast G_2})\cong\mathcal{V}M(G_1,N_1)\oplus\mathcal{V}M(G_2,N_2)\oplus D, $$
where $D$ is defined as an abelian group $\pi_1 (\ker(\theta)).$

Now by the above notations we are in a position to state and prove the following theorem.
%---------------------------------------------------------------------------------------------------------------%
\begin{thm}\label{12}
Let $(G_i,N_i)$ be pairs
of groups for $i=1,2$, then

 $ \begin{array}{ll}
 (i)\ \ \ \ M(G_1 \ast G_2,\langle N_1\ast N_2\rangle^{G_1\ast G_2})&\cong M(G_1,N_1)\oplus M(G_2,N_2).\\
 (ii)\  M^{(2)}(G_1\ast G_2,\langle N_1\ast N_2\rangle^{G_1\ast G_2})&\cong M^{(2)}(G_1,N_1)\oplus M^{(2)}(G_2,N_2)\\&\oplus M(G_1,N_1)\otimes\frac{N_2}{[N_2,G_2]}\\
 &\oplus M(G_2,N_2)\otimes \frac{N_1}{[N_1,G_1]}\\
 &\oplus  M(\frac{G_2}{N_2})\otimes\frac{N_1}{[N_1,G_1]}\\
 &\oplus M(\frac{G_1}{N_1})\otimes\frac{N_2}{[N_2,G_2]}\\
 &\oplus (\frac{G_1}{N_1})^{ab}\otimes M(G_2,N_2)\\
 &\oplus (\frac{G_2}{N_2})^{ab}\otimes M(G_1,N_1)\\
 &\oplus Tor(\frac{N_1}{[N_1,G_1]},\frac{N_2}{[N_2,G_2]})\\
 &\oplus Tor\big((\frac{G_1}{N_1})^{ab},\frac{N_2}{[N_2,G_2]}\big)\\
 &\oplus Tor\big((\frac{G_2}{N_2})^{ab},\frac{N_1}{[N_1,G_1]}\big).\end{array}$
\end{thm}
%---------------------------------------------------------------------------------------------------------------%
\begin{proof}
(i) Let $\mathcal{V}$ be the variety of abelian groups. We have $(K_{1_.}\ast K_{2_.})^{ab}\cong K_{1_.}^{ab}\oplus K_{2_.}^{ab}$ since $K_{1_.}$ and $K_{2_.}$ are free groups. Hence in we have \eqref{free} $\ker\big(\beta/\gamma_2(\beta)\big)\cong \ker\big(\alpha_1/\gamma_2(\alpha_1)\big)\oplus\ker\big(\alpha_2/\gamma_2(\alpha_2)\big)$. Therefore the exact sequence \eqref{ker} implies that $$\begin{array}{ll}M(G_1 \ast G_2,\langle N_1\ast N_2\rangle^{G_1\ast G_2})&\cong\pi_1\Big(\ker\big(\beta/\gamma_2(\beta)\big)\Big)\\&\cong \pi_1\Big(\ker\big(\alpha_1/\gamma_2(\alpha_1)\big)\Big)\oplus\pi_1\Big(\ker\big(\alpha_2/\gamma_2(\alpha_2)\big)\Big)
.\end{array}$$

(ii) Let $\mathcal{V}$ be the variety of nilpotent groups of class at most 2. Burns and Ellis (1997) proved the isomorphisms $\ker(\varphi)\cong (K_{1_.})^{ab}\otimes(K_{2_.})^{ab}$ and $\ker(\psi)\cong (L_{1_.})^{ab}\otimes(L_{2_.})^{ab}$. Hence we have $\ker\big(\theta/\gamma_3(\theta)\big)\cong\ker(\frac{\alpha_1}{\gamma_2(\alpha_1)}
\otimes\frac{\alpha_2}{\gamma_2(\alpha_2)})$.

Since $(K_{1_.})^{ab}\otimes(K_{2_.})^{ab}$ and $(L_{1_.})^{ab}\otimes(L_{2_.})^{ab}$ are free abelian simplicial groups, by computing the ranks of free abelian groups in an exact sequence, we can obtain $\ker\big(\theta/\gamma_3(\theta)\big)$, as follows:
$$\begin{array}{ll}\ker(\frac{\alpha_1}{\gamma_2(\alpha_1)}\otimes\frac{\alpha_2}{\gamma_2(\alpha_2)})&
 \cong\ker(\frac{\alpha_1}{\gamma_2(\alpha_1)}\big)\otimes\ker\big(\frac{\alpha_2}{\gamma_2(\alpha_2)})\\
&\oplus\frac{L_{1_.}}{\gamma_2(L_{1_.})}\otimes\ker(\frac{\alpha_2}{\gamma_2(\alpha_2)})\\ &\oplus\ker(\frac{\alpha_1}{\gamma_2(\alpha_1)})\otimes\frac{L_{2_.}}{\gamma_2(L_{2_.}).}\end{array}$$

Now we compute the fundamental group of $\ker(\frac{\alpha_1}{\gamma_2(\alpha_1)}\otimes\frac{\alpha_2}{\gamma_2(\alpha_2)})$. First we obtain $\pi_1\Big(\ker(\frac{\alpha_1}{\gamma_2(\alpha_1)})\otimes\ker(\frac{\alpha_2}{\gamma_2(\alpha_2)})\Big)$ in more details. By Theorem \ref{sg} \eqref{m} we must obtain $H_1\big(N\big(\ker(\frac{\alpha_1}{\gamma_2(\alpha_1)})\otimes\ker(\frac{\alpha_2}{\gamma_2(\alpha_2)})\big)\big)$. Theorem \ref{sg} \eqref{n} and K\"{u}nneth formula imply that
$$\begin{array}{ll}
 \pi_1\big(\ker(\frac{\alpha_1}{\gamma_2(\alpha_1)})\otimes\ker(\frac{\alpha_2}{\gamma_2(\alpha_2)})\big)
 &\cong \pi_1\big(\ker(\frac{\alpha_1}{\gamma_2(\alpha_1)})\big)\otimes \pi_0 \big(\ker(\frac{\alpha_2}{\gamma_2(\alpha_2)})\big)\\
 &\oplus \pi_0\big(\ker(\frac{\alpha_1}{\gamma_2(\alpha_1)})\big)\otimes \pi_1 \big(\ker(\frac{\alpha_2}{\gamma_2(\alpha_2)})\big)\\
 &\oplus Tor\Big(\pi_0\big(\ker(\frac{\alpha_1}{\gamma_2(\alpha_1)})\big), \pi_0 \big(\ker(\frac{\alpha_2}{\gamma_2(\alpha_2)})\big)\Big)\\
 &\cong M(G_1,N_1)\otimes\frac{N_2}{[N_2,G_2]}\\
 &\oplus \frac{N_1}{[N_1,G_1]}\otimes M(G_2,N_2)\\
 &\oplus Tor(\frac{N_1}{[N_1,G_1]},\frac{N_2}{[N_2,G_2]}).\end{array}$$
 Similarly we have
 $$\begin{array}{ll}\pi_1\big(\frac{L_{1_.}}{\gamma_2(L_{1_.})}\otimes\ker(\frac{\alpha_2}{\gamma_2(\alpha_2)})\big)
 &\cong M(\frac{G_1}{N_1})\otimes\frac{N_2}{[N_2,G_2]}\\
 &\oplus (\frac{G_1}{N_1})^{ab}\otimes M(G_2,N_2)\\
 &\oplus Tor\big((\frac{G_1}{N_1})^{ab},\frac{N_2}{[N_2,G_2]}\big).\end{array}$$
 Also
 $$\begin{array}{ll}\pi_1\big(\frac{L_{2_.}}{\gamma_2(L_{2_.})}\otimes\ker(\frac{\alpha_1}{\gamma_2(\alpha_1)})\big)
 &\cong  M(\frac{G_2}{N_2})\otimes\frac{N_1}{[N_1,G_1]}\\
 &\oplus (\frac{G_2}{N_2})^{ab}\otimes M(G_1,N_1)\\
 &\oplus Tor\big((\frac{G_2}{N_2})^{ab},\frac{N_1}{[N_1,G_1]}\big).\end{array}$$
Now by replacing the above isomorphisms in \eqref{ker}, we conclude the following isomorphism:
 $$\begin{array}{ll}
  M^{(2)}(G_1\ast G_2,\langle N_1\ast N_2\rangle^{G_1\ast G_2})&\cong M^{(2)}(G_1,N_1)\oplus M^{(2)}(G_2,N_2)\\
 &\oplus M(G_1,N_1)\otimes\frac{N_2}{[N_2,G_2]}\\
 &\oplus M(G_2,N_2)\otimes\frac{N_1}{[N_1,G_1]} \\
 &\oplus  M(\frac{G_2}{N_2})\otimes\frac{N_1}{[N_1,G_1]}\\
 &\oplus M(\frac{G_1}{N_1})\otimes\frac{N_2}{[N_2,G_2]}\\
 &\oplus (\frac{G_1}{N_1})^{ab}\otimes M(G_2,N_2)\\
 &\oplus (\frac{G_2}{N_2})^{ab}\otimes M(G_1,N_1)\\
 &\oplus Tor(\frac{N_1}{[N_1,G_1]},\frac{N_2}{[N_2,G_2]})\\
 &\oplus Tor\big((\frac{G_1}{N_1})^{ab},\frac{N_2}{[N_2,G_2]}\big)\\
 &\oplus Tor\big((\frac{G_2}{N_2})^{ab},\frac{N_1}{[N_1,G_1]}\big).\end{array}$$
\end{proof}
%----------------------------------------------------------------------------------------%
\begin{rem}
Theorems 4.1 and 3.2 imply that $M(G_1 \ast G_2)\cong M(G_1)\oplus M(G_2)$ and $M^{(2)}(G\ast H)\cong M^2(G)\oplus M^2(H)\oplus M(G)\otimes H^{ab}\oplus G^{ab}\otimes M(H)\oplus Tor(G^{ab},H^{ab})$ which are proved by Miller (1952), and by Burns and Ellis (1997), respectively.
Also note that the part $(i)$ of the above theorem is proved by Mirebrahimi and Mashayekhy.
\end{rem}
Now we intend to compute $M^{(c)}(G_1\ast G_2,\langle N_1\ast N_2\rangle^{G_1\ast G_2})$, for all $c\geq 1$, with some conditions.
%-----------------------------------------------------------------------------------------%
\begin{thm}
Let $(G_i,N_i)$ be pairs of groups for $i=1,2$ such that $G_1/N_1$ and $G_2/N_2$ satisfy in the following conditions:
$$\begin{array}{llll} &{\frac{G_1}{N_1}}^{ab}\otimes {\frac{G_2}{N_2}}^{ab}&= M^{(1)}({\frac{G_1}{N_1}})\otimes M^{(1)}({\frac{G_2}{N_2}})&=Tor({\frac{G_1}{N_1}}^{ab},{\frac{G_2}{N_2}}^{ab})\\
=&{\frac{G_1}{N_1}}^{ab}\otimes H_3(\frac{G_2}{N_2})&=M^{(1)}({\frac{G_1}{N_1}})\otimes {\frac{G_2}{N_2}}^{ab}&=Tor\big({\frac{G_1}{N_1}}^{ab},M^{(1)}({\frac{G_2}{N_2}})\big)\\
  =& {\frac{G_2}{N_2}}^{ab}\otimes H_3({\frac{G_1}{N_1}})&=M^{(1)}({\frac{G_2}{N_2}})\otimes {\frac{G_1}{N_1}}^{ab} &=Tor\big({\frac{G_2}{N_2}}^{ab},M^{(1)}({\frac{G_1}{N_1}})\big)=0.\end{array}$$
  Also, let for $G_1$ and $G_2$ the following conditions hold:
$$G_1^{ab}\otimes G_2^{ab}=M^{(1)}(G_1)\otimes G_2^{ab}=M^{(1)}(G_2)\otimes G_1^{ab}=Tor(G_1^{ab},G_2^{ab})=0.$$
Then for all $c\geq 1$, we have the following isomorphism:
$$M^{(c)}(G_1\ast G_2,\langle N_1\ast N_2\rangle^{G_1\ast G_2})\cong M^{(c)}(G_1,N_1)\oplus M^{(c)}(G_2,N_2).$$
\end{thm}
%-------------------------------------------------------------------------------------------------%
\begin{proof}
Consider the assumption, like the beginning of the section and let $\mathcal{V}$ be variety of nilpotent groups of class at most c, also we note $\varphi_{\gamma_c}$ by $\varphi_c$ and $\psi_{\gamma_c}$ by $\psi_c$ in briefly.

 The commutative diagram (\ref{dia}) implies the following commutative diagram:
\begin{equation}\label{mc}\tiny{\begin{array}{ccccc}
\pi_2(\ker{\psi_c})&\rightarrow& \pi_1(\ker\theta)&\rightarrow&\pi_1(\ker{\varphi_c})
\\  \downarrow &&\downarrow &&\downarrow\\
  \pi_2(\frac{L_{1_.}\ast L_{2_.}}{\gamma_c(L_{1_.}\ast L_{2_.})})&\rightarrow& \pi_1\big(\ker(\frac{\beta}{\gamma_c(\beta)})\big)&\rightarrow& \pi_1(\frac{K_{1_.}\ast K_{2_.}}{\gamma_c(K_{1_.}\ast K_{2_.})}) \\
 \downarrow &&\downarrow &&\downarrow\\
  \pi_2(\frac{L_{1_.}}{\gamma_c(L_{1_.})})\oplus\pi_2(\frac{L_{2_.}}{\gamma_c(L_{2_.})})&\rightarrow&
 \pi_1\big(\ker(\frac{\alpha_{1_.}}{\gamma_c(\alpha_{1_.})})\big)\oplus
 \pi_1\big(\ker(\frac{\alpha_2}{\gamma_c(\alpha_2)})\big) &\rightarrow&
 \pi_1(\frac{K_{1_.}}{\gamma_c(K_{1_.})})\oplus\pi_1(\frac{K_{2_.}}{\gamma_c(K_{2_.})})
     \end{array}}\end{equation}
The assumption
$$G_1^{ab}\otimes G_2^{ab}=M^{(1)}(G_1)\otimes G_2^{ab}=M^{(1)}(G_2)\otimes G_1^{ab}=Tor(G_1^{ab},G_2^{ab})=0$$
implies that $\pi_1(\ker{\varphi_c})$ is a trivial group (See Vasagh, Mirebrahimi and Mashayekhy). Now like Vasagh, Mirebrahimi and Mashayekhy, by induction on $c$ we prove that the other assumptions yield that $\pi_2(\ker(\psi_c))$ is trivial. Note that $\ker{\psi}_c$ satisfies in the following exact sequence
$$1\rightarrow \frac{[L_{1_.},L_{2_.},_{c-2}F]^F}{[L_{1_.},L_{2_.},_{c-1}F]^F} \rightarrow \ker{\psi}_c\rightarrow\ker{\psi}_{c-1}\rightarrow 1,$$
where $F_.=L_{1_.}*L_{2_.}$. Moreover
 $$ \frac{[L_{1_.},L_{2_.},_{c-2}F]^F}{[L_{1_.},L_{2_.},_{c-1}F]^F}\cong\oplus\sum_{\substack{for\ some\ i+j=c}} {\underbrace{L_{1_.}^{ab}\otimes...\otimes L_{1_.}^{ab}}_{i-times}} \otimes\displaystyle{\underbrace{L_{2_.}^{ab}\otimes...\otimes L_{2_.}^{ab}}_{j-times}}.$$
For $c=2$, we show that $\ker\psi_2\cong L_{1_.}^{ab}\otimes L_{2_.}^{ab}$. Theorem \ref{sg} \eqref{m}, \eqref{n} and Kunneth formula imply that $$\pi_0(\ker\psi_2)\cong\pi_0( L_{1_.}^{ab}\otimes L_{2_.}^{ab})\cong (G_1/N_1)^{ab}\otimes (G_2/N_2)^{ab}=0.$$ Similarly $$\begin{array}{ll}\pi_1(\ker\psi_2)&\cong(G_1/N_1)^{ab}\otimes M^{(1)}(G_2/N_2)\\& \oplus  M^{(1)}(G_1/N_1)\otimes (G_2/N_2)^{ab}\\&\oplus Tor\big((G_1/N_1)^{ab},(G_2/N_2)^{ab}\big)=0.\end{array} $$ Also
 $$\begin{array}{ll} \pi_2(\ker\psi_2)&\cong\pi_2( L_{1_.}^{ab}\otimes L_{2_.}^{ab})\\&\cong (G_1/N_1)^{ab}\otimes H_3(G_2/N_2)\\&\oplus (G_2/N_2)^{ab}\otimes H_3(G_1/N_1)\\&\oplus M^{(1)}(G_1/N_1)\otimes (G_2/N_2)^{ab}\\&\oplus M^{(1)}(G_2/N_2)\otimes (G_1/N_1)^{ab}\\& \oplus Tor\big((G_1/N_1)^{ab},M^{(1)}(G_2/N_2)\big)\\&\oplus  Tor\big((G_2/N_2)^{ab},M^{(1)}(G_1/N_1)\big)=0.\end{array} $$

For $c>2$, we have
\[\begin{array}{ll} &\ \ \ \pi_2({\underbrace{L_{1_.}^{ab}\otimes...\otimes L_{1_.}^{ab}}_{i-times}} \otimes\displaystyle{\underbrace{L_{2_.}^{ab}\otimes...\otimes L_{2_.}^{ab}}_{j-times}})\\
&\cong \pi_2(L_{1_.}^{ab}\otimes L_{2_.}^{ab})\otimes  \pi_0({\underbrace{L_{1_.}^{ab}\otimes...\otimes L_{1_.}^{ab}}_{(i-1)-times}} \otimes\displaystyle{\underbrace{L_{2_.}^{ab}\otimes...\otimes L_{2_.}^{ab}}_{(j-1)-times}})\\
 &\oplus \pi_1(_{1_.}^{ab}\otimes L_{2_.}^{ab})\otimes \pi_1({\underbrace{L_{1_.}^{ab}\otimes...\otimes L_{1_.}^{ab}}_{(i-1)-times}} \otimes\displaystyle{\underbrace{L_{2_.}^{ab}\otimes...\otimes L_{2_.}^{ab}}_{(j-1)-times}})\\
 &\oplus \pi_0(L_{1_.}^{ab}\otimes L_{2_.}^{ab})\otimes  \pi_2({\underbrace{L_{1_.}^{ab}\otimes...\otimes L_{1_.}^{ab}}_{(i-1)-times}} \otimes\displaystyle{\underbrace{L_{2_.}^{ab}\otimes...\otimes L_{2_.}^{ab}}_{(j-1)-times}})\\
  &\oplus Tor\big(\pi_0(L_1^{ab}\otimes L_{2_.}^{ab}), \pi_1({\underbrace{L_{1_.}^{ab}\otimes...\otimes L_{1_.}^{ab}}_{(i-1)-times}} \otimes\displaystyle{\underbrace{L_{2_.}^{ab}\otimes...\otimes L_{2_.}^{ab}}_{(j-1)-times}})\big)\\ &\oplus Tor\big(\pi_1(L_1^{ab}\otimes L_{2_.}^{ab}), \pi_0({\underbrace{L_{1_.}^{ab}\otimes...\otimes L_{1_.}^{ab}}_{(i-1)-times}} \otimes\displaystyle{\underbrace{L_{2_.}^{ab}\otimes...\otimes L_{2_.}^{ab}}_{(j-1)-times}})\big)\\
  &\cong 0. \end{array}\]
 Thus $\pi_2(\ker{\psi}_c)=0$ and by \eqref{mc} we have $\pi_1(\ker(\theta))=0$. Hence
 $M^{(c)}(G_1\ast G_2,\langle N_1~\ast~ N_2\rangle^{G_1\ast G_2})\cong M^{(c)}(G_1,N_1)\oplus M^{(c)}(G_2,N_2)$, for all $c\geq 1$.
 \end{proof}
 \begin{cor}\ \ \ \\

 $(i)$ Let $G_1$ and $G_2$ be two finite groups with $\big(|G^{ab}|,|H^{ab}|\big)=1$,
then for all $c\geq1$
$$M^{(c)}(G_1\ast G_2,\langle N_1\ast N_2\rangle^{G_1\ast G_2})\cong M^{(c)}(G_1,N_1)\oplus M^{(c)}(G_2,N_2).$$

$(ii)$ Let $G_1$ and $G_2$ be two perfect groups such that $M^{(1)}({\frac{G_1}{N_1}})\otimes M^{(1)}({\frac{G_2}{N_2}})$ is trivial, then for all $c\geq1$
$$M^{(c)}(G_1\ast G_2,\langle N_1\ast N_2\rangle^{G_1\ast G_2})\cong M^{(c)}(G_1,N_1)\oplus M^{(c)}(G_2,N_2).$$

 \end{cor}
%---------------------------------------------------------------------------------------%

\end{document}